\UseRawInputEncoding

\documentclass[12pt]{amsart}
\usepackage{amsfonts}
\usepackage{amsmath}
\usepackage{amssymb,latexsym}
\usepackage[mathcal]{eucal}
\usepackage[pdftex,bookmarks,colorlinks,breaklinks]{hyperref}

\input xy
\xyoption{all}

\newcommand{\Bcal}{\mathcal{B}}

\newcommand{\Dcal}{\mathcal{D}}

\newcommand{\Kcal}{\mathcal{K}}

\newcommand{\Mcal}{\mathcal{M}}

\newcommand{\Rcal}{\mathcal{R}}

\newcommand{\Wcal}{\mathcal{W}}
\newcommand{\Xcal}{\mathcal{X}}
\newcommand{\Ycal}{\mathcal{Y}}
\newcommand{\Zcal}{\mathcal{Z}}

\newcommand{\Z}{\mathbb{Z}}

\newcommand{\Qb}{\mathbf{Q}}

\newcommand{\Wb}{\mathbf{W}}
\newcommand{\Xb}{\mathbf{X}}
\newcommand{\Yb}{\mathbf{Y}}
\newcommand{\Zb}{\mathbf{Z}}

\newcommand{\XmuG}{(X,\mathcal{X},\mu,G)\ }
\newcommand{\YnuG}{(Y,\mathcal{Y},\nu,G)\ }
\newcommand{\ZzetaG}{(Z,\mathcal{Z},\zeta,G)\ }

\newcommand{\al}{\alpha}

\newcommand{\del}{\delta}
\newcommand{\Del}{\Delta}

\newcommand{\la}{\lambda}

\newcommand{\tet}{\theta}

\newcommand{\Om}{\Omega}
\newcommand{\ka}{\kappa}

\newcommand{\ol}{\overline}

\newcommand{\br}{\vspace{3 mm}}
\newcommand{\imp}{\Rightarrow}

\newcommand{\Id}{{\rm{Id}}}

\newcommand{\supp}{{\rm{supp\,}}}

\newcommand{\Homeo}{{\rm Homeo}}

\swapnumbers
\theoremstyle{plain}
\newtheorem{thm}{Theorem}[section]
\newtheorem{cor}[thm]{Corollary}
\newtheorem{lem}[thm]{Lemma}

\theoremstyle{definition}
\newtheorem{defn}[thm]{Definition}

\newtheorem{rmk}[thm]{Remark}

\newtheorem{rem}[thm]{Remark}

\newtheorem{exa}[thm]{Example}

\newtheorem{question}[thm]{Question}

\begin{document}


\title[On intermediate factors of a product of disjoint systems]
{On intermediate factors of a product of disjoint systems}

\author{Eli Glasner and Benjamin Weiss}

\address{Department of Mathematics\\
     Tel Aviv University\\
         Tel Aviv\\
         Israel}
\email{glasner@math.tau.ac.il}

\address {Institute of Mathematics\\
 Hebrew University of Jerusalem\\
Jerusalem\\
 Israel}
\email{weiss@math.huji.ac.il}


\setcounter{secnumdepth}{2}



\setcounter{section}{0}


\subjclass[2010]{Primary 37A15, 37A35, 37B05}
%

\begin{abstract}
We consider an intermediate factor situation in two categories:
probability measure preserving ergodic theory and compact topological dynamics.
In the first we prove a master-key  theorem and examine a wide range of applications.
In the second we treat the case when one of the systems is distal and then provide some counterexamples.
\end{abstract}

\keywords{Intermediate factors, quasi-factors, disjointness}


\begin{date}
{June 16, 2025}
\end{date}

\maketitle


\setcounter{secnumdepth}{2}


\setcounter{section}{0}


\section{Introduction}
Let us begin by reminding the reader what disjointness of dynamical systems is and how it arose in
Furstenberg's seminal paper  \cite{Fur3}.
We quote from the introduction to the book \cite{Gl-03}:

\begin{quote}
In the ring of integers $\Z$ two integers $m$ and $n$ have no
common factor if whenever $k|m$ and $k|n$ then $k=\pm 1$.
They are disjoint if  $m\cdot n$ is the least common multiple of
$m$ and $n$.
Of course in $\Z$ these two notions coincide. In his seminal paper
of 1967 \cite{Fur3},
H. Furstenberg introduced the same notions in the
context of dynamical systems, both measure-preserving
transformations and homeomorphisms
of compact spaces, and asked whether in these categories as well
the two are equivalent.
The notion of a factor in, say the measure category, is the natural
one: for an acting group $G$ the dynamical system
$\Yb=(Y,\Ycal,\nu,G)$ is a factor of the dynamical system
$\Xb=(X,\Xcal,\mu,G)$ if there exists a measurable map
$\pi:X \to Y$ with $\pi(\mu)=\nu$ that intertwines the
actions of $G$ on the phase spaces $X$ and $Y$.
A common factor of two systems $\Xb$ and $\Yb$ is thus
a third system $\Zb$ which is a factor of both.
A joining of the two systems $\Xb$ and $\Yb$ is
any system $\Wb$ which admits both as factors and is in turn
spanned by them. According to Furstenberg's definition
the systems $\Xb$ and $\Yb$ are disjoint if the product
system $\Xb\times \Yb$ is the only joining they admit.

\end{quote}
In the ring of integers, we have the property that when $x$ and $y$ are disjoint then
every factor $q$ of $yx$ that admits $y$ as a factor has the form $q = yz$ with $z$ a factor of $x$.
Our goal in this paper is to examine the validity of this property in the class of dynamical systems.

Let $G$ be a topological group.
We examine the following type of questions:

\begin{quote}
Let $\Xb = \XmuG$ and $\Yb=\YnuG$ be two ergodic probability measure preserving dynamical systems.
Suppose also that $\Yb$ is disjoint from every ergodic quasi-factor of $\Xb$.
Let $\Qb = (Q, \eta, G)$ be an intermediate factor
\begin{equation*}
\Yb \times \Xb  \to \Qb \to \Yb.
\end{equation*}
We then ask, when is it the case that
$\Qb \cong \Yb \times \Zb$ with $\Zb = \ZzetaG$ a factor of $\Xb$ \footnote{
A related problem is the following one that was raised by Jean-Paul Thouvenot.
Suppose that $T$ and $S$ are ergodic measure preserving transformations such that
$T \times T$ is isomorphic to $S \times S$ does it follow that $T$ itself is isomorphic to $S$?
 For some results on this problem see \cite{K-12}, and \cite{Ry-18}.}?
\end{quote}

We prove a master-key theorem (Theorem \ref{main} below)  and examine a wide range of applications.
We also consider a similar intermediate factor situations in the category of topological dynamics.
Here we treat the case when one of the systems is distal and then provide some counterexamples.

\br

The notions of a quasi-factor and a joining quasi-factor of an ergodic dynamical system were introduced in
\cite{Gl-83} and  \cite[Definitions  6.1 and  8.19, respectively]{Gl-03}. It is
treated in details in \cite{Gl-03} in the measure preserving category, and in \cite{Gl} in the topological dynamics category.
 It serves us here as a main tool. We will briefly recall the basic definitions.
 
 For a dynamical system $\Xb=(X,\Xcal,\mu,G)$,
a factor system $\Yb=(Y,\Ycal,\nu,G)$ with a factor map
$\pi: \Xb\to \Yb$, can be viewed as the $G$-invariant subalgebra
$\pi^{-1}(\Ycal)\subset \Xcal$.
One can also retrieve the factor $\Yb$
as a measure-preserving transformation on the space $M(X)$ of
probability measures on $X$ as follows.
Disintegrate the measure $\mu$ along the fibers of
$\pi^{-1}(\Ycal)$,
\begin{equation}\label{eq-bary}
\mu=\int_Y\mu_yd\nu(y)
\end{equation}
and observe that the $G$-invariance of $\mu$ implies that
$g\mu_y=\mu_{g y},\ (g \in G)$.
Denoting by $\phi:Y\to M(X)$ the map $\phi(y)=\mu_y$
and letting $\ka=\phi_*(\nu)$, we see that
$\phi:(Y,\Ycal,\nu,G)\to (M(X),\ka,G)$ is an isomorphism.
The connection with $\mu$ is given by
\eqref{eq-bary} which says that
$\mu$ is the barycenter of $\ka$. 

\br

 \begin{defn}
 A {\bf  quasi-factor}
of $(X,\Xcal,\mu,G)$ is any $G$-invariant measure $\kappa$ on $M(X)$ whose barycenter is $\mu$. 
 It is called a {\bf joining quasi-factor} when  the joining
$$
\ka'=\int_M (\tet\times\del_\tet) \, d\ka(\tet),
$$
of the systems $(X,\mu)$ and $(M,\ka)$, is ergodic.
 \end{defn}
 Thus every factor is canonically represented as a quasi-factor, 
 but in general, an ergodic system may posses quasi-factors which are not factors.
 Like factors, quasifactors
inherit some dynamical properties. E.g. zero entropy
and distality are preserved under a passage to a quasi-factor.
We refer the readers to the above sources for the basic results concerning these notions.
%
We can now state our key theorem.

\br

\begin{thm}\label{main}
Let $G$ be a locally compact second countable topological group.
Let $\Xb = \XmuG$ and $\Yb=\YnuG$
be two ergodic probability measure preserving dynamical systems.
Suppose also that $\Yb$ is disjoint from every ergodic joining quasi-factor of $\Xb$.
Let $\Qb = (Q, \eta, G)$ be an intermediate factor as in the following commutative diagram:
\begin{equation*}
\xymatrix
{
\Yb \times \Xb \ar[dd]_{\pi_Y}\ar[dr]^{\phi}  & \\
 & \Qb \ar[dl]^{\theta}\\
\Yb &
}
\end{equation*}
Then $\Qb \cong \Yb \times \Zb$ with $\Zb = \ZzetaG$ a joining quasi-factor of $\Xb$.
\end{thm}

\br

In Section \ref{sec-ergodic} we will prove our key theorem, examine several applications, 
and then consider the question ``when the quasi-factor $\Zb$ 
is actually a factor of $\Xb$". In Section \ref{sec-distal} we establish a positive result (in the topological 
setup) for distal systems, and in Section \ref{sec-cex} we provide some topological counterexamples. 
We thank Valery Ryzhikov for some helpful remarks.

\br

\section{Ergodic theory}\label{sec-ergodic}

\begin{proof}[Proof of theorem \ref{main}]
Given $q \in Q$ let
$$
\bar{q} : = \{x \in X : \phi(\theta(q), x) = q\},
$$
so that the fiber in $Y \times X$ over the point $q \in Q$ has the form
$$
\phi^{-1}(q) =\{\theta(q)\} \times  \bar{q}.
$$
Now let us denote $\la = \nu \times \mu$ and let
$$
\la = \int \la_q \, d\eta(q)
$$
be the disintegration of $\la$ over $\eta$.
Note that $\la_q$ has the form $\la_q  = \del_{\theta(q)} \times \tilde{\la}_q$ for a measure $\tilde{\la}_q$ on $X$
which is supported on the subset $\bar{q}$ of $X$.
Let $\pi : q \mapsto \tilde{\la}_q$ be the corresponding map
from $Q$ into the space $M(X)$ of Borel probability measures on $X$,
and let $\zeta$  denote the measure on $M(X)$ which is the push-forward of $\eta$
under $\pi$, namely $\zeta = \pi_*(\eta)$.
Here $M(X)$ is endowed with its natural Borel structure and is equipped with the induced $G$-action.
We let $\Zb = \ZzetaG$ denote the ergodic system $(M(X), \zeta, G)$.
The system $\Zb$ is a quasi-factor of $\Xb$, in fact
$$
\int \tilde{\la}_q\, d\eta(q) = (\pi_X)_* \left(\int  \del_{\theta(q)} \times \tilde{\la}_q \, d\eta(q)\right) =
 (\pi_X)_*\left( \int \la_q \, d\eta(q)\right) =  (\pi_X)_*(\la) = \mu.
  $$
Therefore, by assumption, $\Zb$
is also disjoint from $\Yb$. It follows that the ergodic product system $\Yb \times \Zb$ is a factor of the system
$(Q, \eta, G)$ under the factor map $J : q  \mapsto(\theta(q), \tilde{\la}_q)$.
We claim that $J$ is an isomorphism.
To see this suppose $J(q) = J(q')$; then $y := \theta(q) = \theta(q')$ and  $\tilde{\la}_q = \tilde{\la}_{q'}$.
However, we have
$$
\bar{q} \supset \supp(\tilde{\la}_q)=\supp(\tilde{\la}_{q'}) \subset \bar{q'}
$$
and since for $\nu$-almost every $y \in Y$ the sets $\{\bar{q} : \theta(q) =y\}$ form a partition
of the fiber $\pi_Y^{-1}(y)$, it follows that $q = q'$.
Thus indeed $(Q, \eta,G) \cong \Yb \times \Zb$

%
%

We will now show that $\Zb$ is a joining quasi-factor.
By the definition of a joining quasi-factor,
we need to show that the measure
$$
\zeta' := \int_{M(X) \times Z} (z \times \delta_z) \, d \zeta(z)  =  \int \left(\tilde{\la}_q \times \del_{\tilde{\la}_q}\right) d\, \eta(q)
 $$
 is ergodic.
 (Warning: in the integrand  of  the expression $ \int (z \times \delta_z )\,  d \zeta(z)$ the first $z$
 is a measure on $X$.)

 We have shown in Theorem \ref{main} that $\Qb \cong \Yb \times \Zb$, so let
 $\pi_\Zb : \Qb \to \Zb$ denote the corresponding projection, so that $\pi_{\Zb}(y,x) =  \tilde{\la}_{\phi(y,x)}$.
 Define a map
 \begin{gather*}
 \Phi : Y \times X \to M(Y) \times M(X) \times M(Z), \quad \\
 \Phi(y,x) = \del_{\theta(\phi(y,x))} \times \tilde{\la}_{\phi(y,x)} \times \del_{\tilde{\la}_{\phi(y,x)}}
 = \del_{\theta(q)} \times \tilde{\la}_q \times \del_{\tilde{\la}_{q}}.
 \end{gather*}
 This map is clearly a factor map and formally the push forward measure  $\Phi_*(\la)$,
 is a measure on  $M(Y) \times M(X) \times M(Z)$. However, since  $\del_{\theta(q)}$ and $ \del_{\tilde{\la}_{q}}$
 are point masses on $Y$ and $Z$ respectively, this latter measure can be considered
 as a measure in $M(Y) \times M(M(X)) \times M(Z)$. Moreover,
 the value $\tilde{\la}_{\phi(y,x)}$ on the set $\{(y,x) : \phi(y,x) =q\}$ is fixed with value
 $\tilde{\la}_{\phi(y,x)} = \tilde{\la}_q$. Thus we can finally think of the measure
 $\Phi_*(\la)$ as a measure on $M(Y) \times M(X) \times M(Z)$.
 As a push forward of $\la$ it is ergodic.

 We will next check
 that $P_{2,3} ( \Phi_*(\la)) = \zeta'$,
 where $P_{2,3} : M(Y) \times M(X) \times M(Z) \to M(X) \times M(Z)$
 is the natural projection,
 thereby proving the ergodicity of $\zeta'$.

 Indeed, we have
$$
 \Phi_*(\la)
=   \int  \left(\del_{\theta(\phi(y,x))} \times \tilde{\la}_{\phi(y,x)} \times \del_{\tilde{\la}_{\phi(y,x)}} \right)
\, d\la(y,x)
$$
and projecting with $P_{2,3}$ we get
$$
P_{2,3} ( \Phi_*(\la))  =
 \int  \left( \tilde{\la}_{\phi(y,x)} \times \del_{\tilde{\la}_{\phi(y,x)}} \right) \, d\la(y,x) =
 \zeta'.
$$
\end{proof}

\br
%
%

We will use the following notations:
For an amenable group $G$ let
\begin{itemize}
\item
$\Kcal$ denote the class of K-systems.
\item
$\Zcal$ the class of zero entropy ergodic systems.
\item
$\Wcal$ the class of weakly mixing systems.
\item
$\Dcal$ the class of ergodic distal systems.
\item
$\Mcal\Mcal$ the class of mildly mixing systems.
\item
$\Rcal$ the class of ergodic rigid systems.
\end{itemize}

\br

It is well known that for an amenable group $G$ the following relations hold:

$\Kcal\, \bot \, \Zcal$,  $\Wcal \, \bot \, \Dcal$ and $\Mcal\Mcal \, \bot \, \Rcal$ (see \cite{LPT},  \cite{Gl-03}).
Also the classes $\Zcal$ and $\Rcal$ are closed under quasi-factors
and the class $\Dcal$ is closed under joining quasi-factors
(see \cite{GW-95} and  \cite[Theorem 10.19]{Gl-03}).
Thus, we have the following:

\begin{cor}
When $G$ is amenable and
\begin{itemize}
\item
$\Yb \in \Kcal$ and $\Xb \in \Zcal$,
\item
$\Yb \in \Wcal$ and $\Xb \in \Dcal$,
\item
$\Yb \in \Mcal\Mcal$ and $\Xb \in \Rcal$,
\end{itemize}
then our theorem holds.
\end{cor}

\br

We now pose the following:

\begin{question}
Let $\Qb \cong \Yb \times \Zb$ be as in Theorem \ref{main},
when is $\Zb$ necessarily a factor of $\Xb$ ?

\end{question}

The following table sums up what we know regarding this question for the classes mentioned above.

(i) $\Yb \in \Kcal$ and $\Xb \in \Zcal$  $\imp$ Yes.

(ii) $\Yb \in \Wcal$ and $\Xb \in \Dcal$  $\imp$ Yes.

(iii) $\Yb \in \Mcal\Mcal$ and $\Xb \in \Rcal$  $\imp$ Yes.

(iv) $\Yb \in \Zcal$ and $\Xb \in \Kcal$  $\imp$ No.

(v) $\Yb \in \Dcal$ and $\Xb \in \Wcal$  $\imp$ No.

(vi) $\Yb \in \Rcal$ and $\Xb \in \Mcal\Mcal$  $\imp$ No.

\br

To justify (i) note that in this case $\Xb$ coincides with the Pinsker factor of $\Yb \times \Xb$
and as $\Zb$, a quasi-factor of $\Xb$, has zero entropy, it follows that $\Xb \to \Zb$.

A similar argument applies for the claim (ii), where now we use the
fact that $\Xb$ is the largest distal factor of $\Yb \times \Xb$, and the
fact that $\Zb$, a quasi-factor of $\Xb$, is also distal, hence a factor of $\Xb$.

Finally for (iii) (in the case of $\Z$-actions), recall that the system $\Yb$ is mildly mixing iff for every IP-sequence $\{n_\al\}$
there is a sub-IP-sequence $\{n_\beta\}$ along which $S^{n_\al} \to   \int \cdot \,d\nu$ on
$L^2(\nu)$ (this follows easily from \cite[Proposition 9.22]{Fur}).
If we are now given a rigid function $f \in L^\infty(\Yb \times \Xb)$, say
$(S \times T)^{n_\al}f \to f$ in $L^2(\nu \times \mu)$ for an IP-sequence $\{n_\al\}$,
 then,  in the direct product $\Yb \times \Xb$, along an appropriate
sub-IP-sequence $\{n_\beta\}$ we see that
$f = \lim (S \times T)^{n_\beta} f$ is $\Xcal$-measurable.
It then follows that every $\Zcal$-measurable function in $L^\infty(\nu \times \mu)$ is
$\Xcal$-measurable, hence here also $\Zb$ is a factor of $\Xb$.

\br

\br

Just one counterexample will justify the negative claims (iv), (v) and (vi)
(and with $\Qb = \Yb \times \Xb$).
(Note however, that the classes $\Kcal$, $\Wcal$ and $\Mcal\Mcal$ are, in general, not closed
under passage to quasi-factors.)

This example is basically due to Dan Rudolph and the implication to our
setup was already noted by Kenneth Berg, \cite{Berg}.
We thank Jean-Paul Thouvenot for helpful discussions clarifying the details of this example.
For completeness we will next explain this implication.

\begin{exa}
There exist two  $K$-automorphism $\Xb_i,\ i=1,2$, neither one is is a factor of the other,
and a zero entropy system $\Yb$
such that $\Yb \times \Xb_1 \cong \Yb \times \Xb_2$.
\end{exa}

\br

We see this as follows.  For a measure preserving system $\Xb = (X, \Xcal, \mu, T)$
let $\hat{T} : X \times \{0,1\} \to  X \times \{0,1\}$ be defined by
$$
\hat{T}(x,0) = (x,1) \quad {\text{and}} \quad \hat{T}(x,1) = (Tx,0).
$$
Note that $\hat{T}^2 = T \times \Id$.
Also, for a measure preserving system $\Yb = (Y, \Ycal, \nu, S)$
let $\tilde{S} : Y \times \{0,1\} \to  Y \times \{0,1\}$ be defined by
$$
\tilde{S}(y,0) = (Sy,1) \quad {\text{and}} \quad \tilde{S}(y,1) = (Sy,0),
$$
so that $\tilde{S} = S \times$ flip.

\begin{lem}\label{S^2}
The transformation $T$ has a square root $S$ (i.e. $T = S^2$) iff $\hat{T} \cong \tilde{S}$.
\end{lem}

\begin{proof}
Suppose first that $T = S^2$. Define $\theta : X \times \{0,1\} \to  X \times \{0,1\}$ by
$$
\theta(x,0) = (x,1) \quad {\text{and}} \quad \theta(x,1) = (Sx,0).
$$
Then $\tilde{S} \circ \theta = \theta \circ \hat{T}$, so that $\theta$ defines an isomorphism between $\hat{T}$and $\tilde{S}$.

Conversely, if $\hat{T}$ is isomorphic to some
$\tilde{S} = S \times $ flip, then $\hat{T}^2 \cong \tilde{S}^2$.
Since $\tilde{S}^2(y,i) = (S^2y, i)$ and $\hat{T}^2(x,i) = (Tx, i)$ it follows that $T \cong S^2$, so that $T$ has a square root.
\end{proof}
In \cite{Ru-76} Rudolph constructs two non-isomorphic $K$-automorphisms $T_1, T_2$ such that $T_1^2 =  T_2^2$.
On a close examination of his proof it can be checked  that
neither $T_1$ nor $T_2$ is a factor of the other transformation.
Let us denote $T = T_1^2 =  T_2^2$ and then deduce from Lemma \ref{S^2} that
$\hat{T} \cong \tilde{S}_1 = S_1 \times {\text{flip}} = S_2 \times {\text{flip}} = \tilde{S}_2$.

For our example we now take $\Yb = (\{0,1\}, {\text{flip}})$ and $\Xb_i = (X, \Xcal, \mu, S_i), \ i=1,2$.
\qed

\begin{rmk}
In his work \cite{Th-75} Thouvenot has shown that
any factor of a system of the form
      Bernoulii $\times$ zero entropy
is again of the form
     Bernoulli $\times$ zero entropy.
     We can use a similar argument to show that this result can not be extended to the class of $K$-automorphisms.
     To see this we first observe that for any measure preserving transformation $T$, the product
     transformation $T \times T$ always has a root, namely with
     $$
     R(x,x') =(Tx',x)
     $$
     we have $R^2 = T \times T$.
Now let $T$ be a $K$-automorphism with no square root. Then $T$ is a factor of $T \times T$ and thus also
$\hat{T}$ is a factor of $\widehat{T \times T}$. Since $T\times T$ has a square root
$\widehat{T \times T} \cong \tilde{R} = R \times {\text{flip}}$ --- a product of a $K$-automorphism with the flip,
a zero entropy system. However its factor $\hat{T}$ is not of this form.
\qed
\end{rmk}

\br

\section{Topological dynamics; the distal case}\label{sec-distal}
A {\em topological dynamical system}
is a pair $\Xb =(X,G)$ where $X$ is a compact space and the group $G$
acts on $X$ via a homomorphism of $G$ into
the group $\Homeo(X)$ of self homeomorphisms of $X$.
Unless we say otherwise we assume that our systems are metrizable.
A factor map $\pi : \Xb \to \Yb = (Y,G)$ between two such systems
is a continuous surjective map satisfying $g \circ \pi = \pi \circ g$ for every $g \in G$.
The system $\Xb$ is {\em minimal} when every $G$-orbit $\{gx :  g \in G\}$ is dense.
Two minimal systems $\Xb$ and $\Yb$ are {\em disjoint} if the product system
$(X \times Y, G)$ (with diagonal action) is minimal.
%
A pair of points $(x_1,x_2) \in X \times X$ is {\em proximal} if the orbit closure
$\ol{\{g(x_1,x_2) : g \in G\}}$ meets the diagonal $\Del_X = \{(x,x) : x \in X\}$.
The collection $P \subset X \times X$ of all proximal pairs is
called the {\em proximal relation}. The system $\Xb$ is {\em distal}
when $P = \Del_X$.

The {\em enveloping semigroup}  $E(X,G)$ of $(X,G)$ is the closure in the product space
$X^X$ of the collection $\{\Breve{g} : g \in G\}$, where $\Breve{g}$ is the image of $g$ in $\Homeo(X)$.
This is both a compact right-topological-semigroup and a $G$  dynamical system
(note however that the topological space $E(X,T)$ is usually non-metrizable).
It has a rich topological and algebraic structure
and serves as an important tool in studying the asymptotic features of a dynamical system.
A theorem of Ellis, \cite{E-58} asserts that a minimal system $(X,G)$ is distal if and only if $E(X,G)$ is a group.
It follows that when $(X,G)$ is minimal and distal then the dynamical system $(E(X,G), G)$ is also
minimal and distal.

With a minimal dynamical system $(X,G)$ there is a naturally associated system $(2^X,G)$ on the compact
space $2^X$ comprising the closed subsets of $X$. A minimal subsystem $Z \subset 2^X$ is called a {\bf
quasi-factor} of $(X,G)$.
We have the following theorem from \cite[Theorem 2.5]{Gl-75}.

\begin{thm}\label{QF}
A quasi-factor of a minimal distal system $(X,G)$ is a factor of the dynamical system $(E(X,G), G)$.
In particular, every quasi-factor of a minimal distal system is distal.
\end{thm}

We can now state and prove the main theorem of this section.

\begin{thm}
Consider the following commutative diagram:
\begin{equation*}
\xymatrix
{
\Yb \times \Xb \ar[dd]_{\pi_Y}\ar[dr]^{\phi}  & \\
 & \Qb \ar[dl]^{\theta}\\
\Yb &
}
\end{equation*}
where $\Xb=(X,G)$ is minimal and distal and $\Yb = (Y,G)$ is minimal and disjoint from $(E(X,G),G)$
(this is the case e.g. when $(Y,G)$  it is weakly mixing).
Then in this situation $\Qb \cong \Yb \times \Zb$ with $\Zb =(Z,G)$ a factor of $\Xb$.
\end{thm}

\begin{proof}
Given $q \in Q$ let
$$
\bar{q} : = \{x \in X : \phi(\theta(q), x) = q\},
$$
so that the fiber in $Y \times X$ over the point $q \in Q$ has the form
$$
\phi^{-1}(q) =\{\theta(q)\} \times  \bar{q}.
$$
The map $q \mapsto \bar{q}$ is upper-semi-continuous; i.e. for a converging sequence $Q \ni q_i \to q$ in $Q$
we have $\limsup \bar{q}_i \subset \bar{q}$.
In fact, if $\bar{q}_i \ni x_i \to x$ then the equation
$$
q \leftarrow q_i =\phi(\theta(q_i), x_i)  \to q(\theta(q), x),
$$
shows that $x \in   \bar{q}$.
It now follows that there is a dense $G_\delta$ invariant set $Q_0 \subset Q$ where the map
$q \mapsto \bar{q}$ is continuous.

We define subsets $Z \subset 2^X$ and $W \subset Q \times Z$ as follows:
$$
Z = \ol {\{\bar{q} : q \in Q_0\}}, \quad
W = \ol {\{(q, \bar{q}) : q \in Q_0\}}.
$$

Standard arguments (see \cite{Gl-75}) now show that
\begin{itemize}
\item
$Z$ and $W$ are minimal systems (thus $Z$ is a quasi-factor of $X$).
\item
The projection map $\pi_Q : W \to Q, \ (q, z) \mapsto q$ is an almost one-to-one
homomorphism from $W$ onto $Q$.
\end{itemize}
Next apply Theorem \ref{QF} to deduce that $Z$ is a distal system and moreover is a factor
of $E(X,G)$.

Suppose $q =\pi_Q(q,z_1)= \pi_Q(q,z_2)$, then the points $z_1, z_2$ are both distal
and proximal, hence $z_1 = z_2$. Thus the map $\pi_Q$ is in fact an isomorphism:
$Q \cong W$.

Finally, as by assumption $Y$ is disjoint fron $E(X,G)$, it is also disjoint from $Z$.
Because both $Y$ and $Z$ are factors of $W$ we deduce that $W \to Y \times Z$.
But, clearly the factor maps $\pi_Q$ and $\pi_Y : W \to Y$ separate the points on $W$, whence
$W \cong Q \cong Y \times Z$ as claimed.
\end{proof}

\br

\section{Topological dynamics; some counter examples}\label{sec-cex}

A {\em cascade} is topological system where the acting group is the integers, generated by a single homeomorphism.
I.e. a pair $\Xb =(X,T)$ where $X$ is a compact metric space and $T : X \to X$ a self-homeomorphism.
%
%
%
%
Two minimal systems $\Xb$ and $\Yb$ are disjoint if the product system
$(X \times Y, T \times S)$ is minimal.
The system $\Xb$ is {\em uniquely ergodic} if there is on $X$ a unique $T$-invariant probability measure.
It has {\em uniformly positive entropy} if each non-diagonal pair $(x_1, x_2) \in X \times X \setminus \Del$ is
an entropy pair.
Every minimal system with zero topological entropy is topologically
disjoint from every minimal system with uniform positive entropy (for more details see e.g. \cite{Gl-03}).

Suppose that $(X,T)$ and $(Y,T)$ are two disjoint minimal dynamical systems.
Is it always true that every intermediate factor $Q$ of the form $Y \times X \xrightarrow[]{}Q\xrightarrow[]{} Y$ is of the form
$Y \times Z$, where $Z$ is a factor of $X$?
The answer is negative, as can be seen e.g. in the following:

\begin{exa}
There exist two, uniquely
ergodic systems, of uniformly positive entropy $\Xb_i,\ i=1,2$,
neither one is a factor of the other,
and a zero entropy system $\Yb$
such that $\Yb \times \Xb_1 \cong \Yb \times \Xb_2$.
\end{exa}

\begin{proof}
We again consider Rudolph's example \cite{Ru-76} $(X, \mathcal{B}, \mu, T_1)$, where
using his notations, $X = \Om \times \{0,1\}$.
Then, with $T' = \Id_\Om \times F$, $F$ denoting the flip on $\{0,1\}$,
we have $T_2 = T' T_1$ (so that $T'^2 = \Id$ and $T' T_1 = T_1 T'$).
Let $G$ be the group of measure preserving transformations of $(X,\Bcal,\mu)$
generated by $T_1$ and $T'$. We then regard $(X, \Bcal,\mu, G)$ as a $\Z \times \Z_2$-system.

Next apply (an extended version of) the Jewett-Krieger theorem to this
$G$-system (see \cite{We-85}) to obtain
a topological minimal, uniquely ergodic model which we denote as $(X,\mu, G)$.
On $X$ we now have the homeomorphisms $T_1$ and $T_2 = T' T_1$, and following the exact same argument
as in the measure theoretical example above, we obtain the required topological example.
Note that the $T_i$, not being measure theoretically factors of each other, also, a fortiori,
have this property as topological systems.
It only remains to observe that by \cite{GW-94} the homeomorphisms $T_1$ and $T_2$ have uniformly positive entropy.
\end{proof}

Another example, with $G = F_2$, the free group on two generators, is as follows.

\begin{exa}
Let $G =F_2 = \langle a, b, a^{-1}, b^{-1} \rangle$. Let
$Y = \hat{F}_2$ be the profinite completion of $F_2$ and let $X = \partial F_2$ be its Gromov boundary (see \cite{DM}).
Recall that in this case the phase space of the flow $X =\partial G$ is the Cantor set
formed by all the infinite reduced words on the symbols $a,b, a^{-1}, b^{-1}$.
Since $Y$ is an isometric flow and $X$ is stronly proximal, we have that $Y$ is disjoint from $X$;
i.e. the product flow $Y \times X$ is minimal.
Let $R_a \subset X \times X$ be the set
$$
R_a =  \{g (a^{\infty},a^{-\infty}), g(a^{-\infty}, a^{\infty}): g \in G\} \cup \Delta_X.
$$
It is easily seen that $R_a$ is a closed invariant equivalence
relation on $X$, corresponding to a factor map $X \to X_a = X/R_a$.
We similarly define $R_b$ and the corresponding factor map
$X \to X_b = X / R_b$.

Next choose points $y_0, y_1 \in Y$ such that $G y_0
\not= G y_1$ and then define $R$ to be the following relation on the $Y \times X$.
\begin{align*}
R = & \{((g y_0,x) , (g y_0,x')) :  g \in G, (x,x') \in R_a\} \cup \\
& \{((g y_1,x) , (g y_1,x')) : g \in G, (x,x') \in R_b\}  \cup\\
& \Delta_{Y \times X}.
\end{align*}
Again it is clear that $R$ is an ICER on $Y \times X$. Let
$Q = Y \times X /R$ be the corresponding factor.
Now check that $Q$ is not of the form $Y \times Z$
for any factor $X \to Z$.
\end{exa}

\begin{rem}
A similar example can be given with $\Z$ as the acting group.
\end{rem}

\br

\end{document}